\documentclass[letterpaper]{amsproc}
\usepackage{amsfonts}
\usepackage{amssymb}
\usepackage{amsmath}
\usepackage{amsthm}
\usepackage{color}
\usepackage{graphicx}
\usepackage{float}
\usepackage[colorlinks=true, citecolor=blue, urlcolor=cyan, linktocpage]{hyperref}
\usepackage{enumerate}

\usepackage{tikz}
\usetikzlibrary{arrows,matrix,positioning}
\usetikzlibrary{calc}
\usepackage{wrapfig}
\usepackage{verbatim}

\newtheorem{theorem}{Theorem}[section]

\newtheorem{conjecture}[theorem]{Conjecture}

\theoremstyle{definition}
\newtheorem{example}[theorem]{Example}
\theoremstyle{definition}

\theoremstyle{plain}

\theoremstyle{definition}

\numberwithin{figure}{section}
\numberwithin{theorem}{section}

\title[Generating B\'ezout trees for Pythagorean pairs]{GENERATING B\'EZOUT TREES FOR PYTHAGOREAN PAIRS}

\author[Emily Gullerud and James S. Walker]{\small EMILY GULLERUD AND JAMES S. WALKER}

\address{
Department of Mathematics \\
University of Wisconsin-Eau Claire}
\email{gullerej@uwec.edu}

\address{ 
Department of Mathematics \\
University of Wisconsin-Eau Claire}
\email{walkerjs@uwec.edu}

\begin{document}

\setcounter{page}{1} \thispagestyle{empty}

\begin{abstract}
Relatively prime pairs of integers can be represented as nodes in three way branching trees. We construct trees of B\'ezout coefficients which correspond to the relatively prime pairs in the aforementioned trees. As one application, we compare the B\'ezout coefficients in these trees to those returned by the \texttt{gcd} function in \textsc{Matlab}. As another application, we use these trees to decrease the computation time required to create computer generated hyperbolic wallpaper designs.
\end{abstract}

\maketitle

\section{Generating the trinary trees}

The pair of integers $(m,n)$ is defined to be a Pythagorean pair if they generate a Pythagorean triplet $(x,y,z)$ where $x^2+y^2=z^2$ for $x,y,z\in \mathbb{Z}$. These integers $m$ and $n$ are relatively prime and satisfy the equations $$x=m^2-n^2, \quad y=2mn, \quad z=m^2+n^2.$$ For example, the pair $(2,1)$ generates the Pythagorean triplet $(3,4,5)$. Not only does $(m,n)$ generate a Pythagorean triplet, but its \textit{associated pairs} $(n,m)$ and $(m,-n)$ generate the triplets $(-x,y,z)$ and $(x,-y,z)$ respectively, which clearly still satisfy the equation $x^2+y^2=z^2$.

Using the fact that $(2m+n,m)$ is a Pythagorean pair given that $(m,n)$ is a Pythagorean pair, Randall and Saunders \cite{Randall1994} proved that every relatively prime pair of integers can be obtained through two trinary trees, one starting with the pair $(2,1)$ and another with the pair $(3,1)$. These are three way branching trees since they use associated pairs to obtain subsequent branches. The construction of these trees is summarized below.

\begin{theorem}[Randall and Saunders, 1994]\label{thm:treenodes}
Let $(m,n)$ be a Pythagorean pair such that $m>n$ with associated pairs $(n,m)$ and $(m,-n)$. Define $f: \mathbb{Z} \times \mathbb{Z} \rightarrow \mathbb{Z} \times \mathbb{Z}$ by $f(m,n) = (2m+n,m)$. Then 
\begin{align*}
f(m,n) &= (2m+n,m),\\
f(n,m) &= (2n+m,n), \text{ and}\\
f(m,-n) &= (2m-n,m)
\end{align*}
are all Pythagorean pairs.
\end{theorem}

Starting with the Pythagorean pair $(m,n)$, applying $f$ as defined in Theorem~\ref{thm:treenodes} to $(m,n)$ and its associated pairs $(n,m)$ and $(m,-n)$ and then successively applying $f$ to the resultant Pythagorean pairs results in a trinary tree. We denote this the \textbf{trinary tree generated by $\mathbf{(m,n)}$}; see Figure~\ref{fig:TrinaryTree}. Randall and Saunders proved that the trinary tree produced from $(3,1)$ contains all pairs of relatively prime odd integers. Similarly, the trinary tree produced from $(2,1)$ contains all pairs of relatively prime integers of opposite parity. Hence using only two trinary trees we can produce all the pairs of relatively prime integers. 

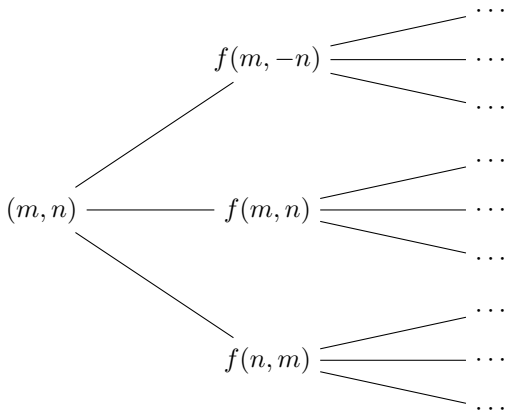
\begin{figure}[h]
\centering
\begin{tikzpicture}
\node at (-3,0) (A){$(m,n)$};
\node at (0,2) (B_1){$f(m,-n)$};
\node at (0,0) (B_2){$f(m,n)$};
\node at (0,-2) (B_3){$f(n,m)$};
\node at (3,2.65) (C_1){$\cdots$};
\node at (3,2) (C_2){$\cdots$};
\node at (3,1.35) (C_3){$\cdots$};
\node at (3,.65) (C_4){$\cdots$};
\node at (3,0) (C_5){$\cdots$};
\node at (3,-.65) (C_6){$\cdots$};
\node at (3,-1.35) (C_7){$\cdots$};
\node at (3,-2) (C_8){$\cdots$};
\node at (3,-2.65) (C_9){$\cdots$};
\draw (A) -- (B_1);
\draw (A) -- (B_2);
\draw (A) -- (B_3);
\draw (B_1) -- (C_1);
\draw (B_1) -- (C_2);
\draw (B_1) -- (C_3);
\draw (B_2) -- (C_4);
\draw (B_2) -- (C_5);
\draw (B_2) -- (C_6);
\draw (B_3) -- (C_7);
\draw (B_3) -- (C_8);
\draw (B_3) -- (C_9);
\end{tikzpicture}
\caption{The trinary tree generated by $(m,n)$.}
\label{fig:TrinaryTree}
\end{figure}

Since the Pythagorean pair $(m,n)$ is relatively prime, we can find a pair of integers $(u,v)$ such that $mu+nv=1$. We show that the pairs $(u,v)$ can be produced recursively in a similar fashion to the pairs $(m,n)$.

\begin{theorem}
Let $(m,n)$ be a Pythagorean pair such that $m>n$ with associated pairs $(n,m)$ and $(m,-n)$. Let $f$ be defined as in Theorem~\ref{thm:treenodes} and let $mu+nv=1$ for some $u,v \in \mathbb{Z}$. Define the function $g: \mathbb{Z} \times \mathbb{Z} \rightarrow \mathbb{Z} \times \mathbb{Z}$ by $g(u,v) = (v,u-2v)$. Then 
\begin{align*}
g(u,v) &= (v,u-2v),\\
g(v,u) &= (u,v-2u),\text{ and}\\
g(u,-v) &= (-v,u+2v)
\end{align*}
 yield the necessary $u'$ and $v'$ to satisfy 
\begin{align*}
(2m+n)u' + mv' &= 1,\\
(2n+m)u' + nv' &= 1,\text{ and}\\
(2m-n)u' + mv' &= 1
\end{align*}
respectively.
\end{theorem}
\begin{proof}
Let $(m,n)$ be a Pythagorean pair with associated pairs $(n,m)$ and $(m,-n)$. Then by Theorem~\ref{thm:treenodes}, we have the Pythagorean pairs
\begin{enumerate}
\item $f(m,n) = (2m+n,m)$
\item $f(n,m) = (2n+m,n)$
\item $f(m,-n) = (2m-n,m)$
\end{enumerate}
Since $(m,n)$ is a Pythagorean pair, we know $\gcd{(m,n)} = 1$. So there exist $u,v \in \mathbb{Z}$ such that $mu + nv = 1$. We will consider each of three cases above separately.

\noindent \textbf{CASE 1:} Consider $f(m,n) = (2m+n,m)$. Since $(2m+n,m)$ is a Pythagorean pair, $\gcd{(2m+n,m)} = 1$. So there exist $u',v' \in \mathbb{Z}$ such that $(2m+n)u' + mv' = 1$. Then
$$(2m+n)u' + mv' = 2mu' + nu' + mv' = (2u'+v')m + u'n = 1.$$
Since $mu + nv = 1$, we get $2u'+v' = u$ and $u' = v$. Hence $u' = v$ and $v' = u-2v$.

\noindent \textbf{CASE 2:} Consider $f(n,m) = (2n+m,n)$. Since $(2n+m,n)$ is a Pythagorean pair, $\gcd{(2n+m,n)} = 1$. So there exist $u',v' \in \mathbb{Z}$ such that $(2n+m)u' + nv' = 1$. Then $$(2n+m)u' + nv' = 2nu' + mu' + nv' = u'm + (2u'+v')n = 1.$$
Since $mu + nv = 1$, we get $u'= u$ and $2u'+v'=v$. Hence $u' = u$ and $v' = v-2u$.

\noindent \textbf{CASE 3:} Consider $f(m,-n) = (2m-n,m)$. Since $(2m-n,m)$ is a Pythagorean pair, $\gcd{(2m-n,m)} = 1$. So there exists $u',v' \in \mathbb{Z}$ such that $(2m-n)u' + mv' = 1$. Then $$(2m-n)u' + mv' = 2mu' - nu' + mv' = (2u'+v')m + (-u')n = 1.$$
Since $mu + nv = 1$, we get $2u'+v' = u$ and $-u' = v$. Hence $u' = -v$ and $v' = u+2v$.

Define $g: \mathbb{Z} \times \mathbb{Z} \rightarrow \mathbb{Z} \times \mathbb{Z}$ by $g(u,v) = (v,u-2v)$. One can verify that this function satisfies the three cases above. 
\end{proof}

We call the trinary tree generated by $(u,v)$ for the corresponding Pythagorean pair $(m,n)$ the \textbf{B\'ezout tree of $\mathbf{(m,n)}$ generated by $\mathbf{(u,v)}$}. Notice that there is not a unique pair $(u,v)$ for each $(m,n)$; in fact, there are infinitely many such pairs. Thus there exist infinitely many B\'ezout trees of $(m,n)$, where each is dependent on the pair chosen to be the root of the tree.

\begin{example}\label{ex:trees}
The trinary tree generated by $(3,1)$ is given below on the left and the B\'ezout tree of $(3,1)$ generated by $(0,1)$ is given below on the right both to a depth of two:
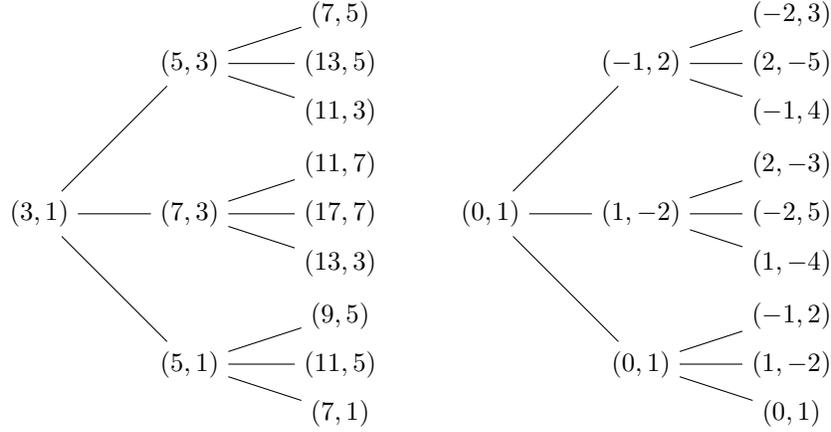
\begin{figure}[h]
\centering
\begin{tikzpicture}
\node at (-5,0) (A){$(3,1)$};
\node at (-3,2) (B_1){$(5,3)$};
\node at (-3,0) (B_2){$(7,3)$};
\node at (-3,-2) (B_3){$(5,1)$};
\node at (-1,2.65) (C_1){$(7,5)$};
\node at (-1,2) (C_2){$(13,5)$};
\node at (-1,1.35) (C_3){$(11,3)$};
\node at (-1,.65) (C_4){$(11,7)$};
\node at (-1,0) (C_5){$(17,7)$};
\node at (-1,-.65) (C_6){$(13,3)$};
\node at (-1,-1.35) (C_7){$(9,5)$};
\node at (-1,-2) (C_8){$(11,5)$};
\node at (-1,-2.65) (C_9){$(7,1)$};
\draw (A) -- (B_1);
\draw (A) -- (B_2);
\draw (A) -- (B_3);
\draw (B_1) -- (C_1);
\draw (B_1) -- (C_2);
\draw (B_1) -- (C_3);
\draw (B_2) -- (C_4);
\draw (B_2) -- (C_5);
\draw (B_2) -- (C_6);
\draw (B_3) -- (C_7);
\draw (B_3) -- (C_8);
\draw (B_3) -- (C_9);
\node at (1,0) (X){$(0,1)$};
\node at (3,2) (Y_1){$(-1,2)$};
\node at (3,0) (Y_2){$(1,-2)$};
\node at (3,-2) (Y_3){$(0,1)$};
\node at (5,2.65) (Z_1){$(-2,3)$};
\node at (5,2) (Z_2){$(2,-5)$};
\node at (5,1.35) (Z_3){$(-1,4)$};
\node at (5,.65) (Z_4){$(2,-3)$};
\node at (5,0) (Z_5){$(-2,5)$};
\node at (5,-.65) (Z_6){$(1,-4)$};
\node at (5,-1.35) (Z_7){$(-1,2)$};
\node at (5,-2) (Z_8){$(1,-2)$};
\node at (5,-2.65) (Z_9){$(0,1)$};
\draw (X) -- (Y_1);
\draw (X) -- (Y_2);
\draw (X) -- (Y_3);
\draw (Y_1) -- (Z_1);
\draw (Y_1) -- (Z_2);
\draw (Y_1) -- (Z_3);
\draw (Y_2) -- (Z_4);
\draw (Y_2) -- (Z_5);
\draw (Y_2) -- (Z_6);
\draw (Y_3) -- (Z_7);
\draw (Y_3) -- (Z_8);
\draw (Y_3) -- (Z_9);
\end{tikzpicture}
\end{figure}
\end{example}

\section{Comparing the \textsc{Matlab} \texttt{gcd} function to trinary trees}

The \texttt{gcd} function in \textsc{Matlab} \cite{Matlab} has the following syntax:
$$\texttt{gcd}(A,B) = [G,U,V]$$ where $G$ is the greatest common divisor of $A$ and $B$, while $U$ and $V$ are the B\'ezout coefficients satisfying $A\cdot U+B\cdot V=G$. Even though there are infinitely many choices for $U$ and $V$, the \texttt{gcd} function will yield only one choice. This is done so using the extended Euclidean algorithm. We now consider the question of \textbf{when does the \texttt{gcd} function yield the same B\'ezout coefficients as the B\'ezout tree?}

Recall that the two trinary trees produced by $(2,1)$ and $(3,1)$ yield all the pairs of relatively prime integers. Hence the B\'ezout trees of $(2,1)$ and $(3,1)$ will produce one pair of B\'ezout coefficients for each pair of relatively prime integers. We have the freedom to choose any pair of B\'ezout coefficients which correspond to $(2,1)$ and $(3,1)$ to be the roots of our B\'ezout trees. In an attempt to match the B\'ezout coefficients produced by the \texttt{gcd} function to those in the B\'ezout trees, we select the following coefficients $(u_r,v_r)$ for our roots:
\begin{itemize}
\item The B\'ezout tree of $(2,1)$: $(u_r,v_r) = (0,1)$ since $\texttt{gcd}(2,1)=[1,0,1]$.
\item The B\'ezout tree of $(3,1)$: $(u_r,v_r) = (0,1)$ since $\texttt{gcd}(3,1)=[1,0,1]$.
\end{itemize}
Notice that the B\'ezout trees of $(2,1)$ and $(3,1)$ defined as above are identical. Recall that we gave this tree to a depth of $2$ in Example~\ref{ex:trees}.

We computed the trinary trees generated by $(2,1)$ and $(3,1)$ and their corresponding B\'ezout trees generated by $(0,1)$ to a depth of $13$. This yielded a total of $4782966$ relatively prime pairs and their corresponding B\'ezout coefficients. We then found the percentage of pairs $(u,v)$ in the B\'ezout trees which were equal to the pairs $(U,V)$ given by the \texttt{gcd} function. Our test showed that exactly one sixth of these pairs differed. Of the pairs which differed, all were generated in the same branch of the B\'ezout tree of $(2,1)$, specifically the branch corresponding to the relatively prime pair $(2,-1)$ found in the second level. In Figure~\ref{fig:BezoutTreeGCD} we show the B\'ezout tree of $(2,1)$ generated by $(0,1)$ with the differing pairs indicated; the pairs from the B\'ezout trees are denoted in blue and the pairs given by the \texttt{gcd} function are denoted in red. Pairs which are the same in both the tree and \texttt{gcd} function are given once in black.

\begin{figure}[h]
\centering
\begin{tikzpicture}
\node at (-3,0) (A){$(0,1)$};
\node[blue] at (0,2.5) (B_1){\footnotesize$(-1,2)$};
\node[red] at (0,2.9) (D_1){\footnotesize$(1,-1)$};
\node at (0,0) (B_2){$(1,-2)$};
\node at (0,-2) (B_3){$(0,1)$};
\node[blue] at (3,3.5) (C_1){\footnotesize$(-2,3)$};
\node[red] at (3,3.9) (D_1){\footnotesize$(1,-1)$};
\node[blue] at (3,2.5) (C_2){\footnotesize$(2,-5)$};
\node[red] at (3,2.9) (D_1){\footnotesize$(1,-1)$};
\node[blue] at (3,1.5) (C_3){\footnotesize$(-1,4)$};
\node[red] at (3,1.9) (D_1){\footnotesize$(1,-1)$};
\node at (3,.65) (C_4){$(2,-3)$};
\node at (3,0) (C_5){$(-2,5)$};
\node at (3,-.65) (C_6){$(1,-4)$};
\node at (3,-1.35) (C_7){$(-1,2)$};
\node at (3,-2) (C_8){$(1,-2)$};
\node at (3,-2.65) (C_9){$(0,1)$};
\draw (A) -- (-.7,2.6);
\draw (A) -- (B_2);
\draw (A) -- (B_3);
\draw (.7,2.85) -- (2.4,3.7);
\draw (.7,2.7) -- (2.4,2.7);
\draw (.7,2.55) -- (2.4,1.7);
\draw (B_2) -- (C_4);
\draw (B_2) -- (C_5);
\draw (B_2) -- (C_6);
\draw (B_3) -- (C_7);
\draw (B_3) -- (C_8);
\draw (B_3) -- (C_9);
\draw[olive,thick,dotted] (-.6,2.3) -- (.6,2.3) -- (.6,2.7) -- (-.6,2.7) -- cycle;
\end{tikzpicture}
\caption{B\'ezout tree of $(2,1)$ generated by $(0,1)$.}
\label{fig:BezoutTreeGCD}
\end{figure}
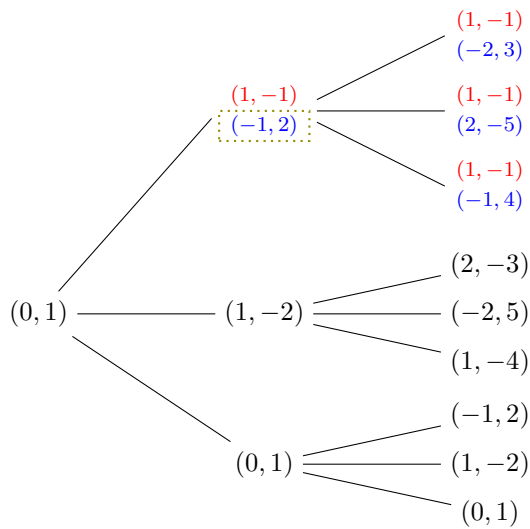

In our B\'ezout tree generated by $(0,1)$, if we change the pair $(-1,2)$ to the pair given by the \texttt{gcd} function, which is $(1,-1)$, then computation to a depth of $13$ shows that the B\'ezout tree will have the same exact pairs as the \texttt{gcd} function. This leads to the following conjecture, which we strongly believe to be true given our data.

\begin{conjecture}\label{conj:BezoutTreeGCD}
Consider the trinary trees generated by $(2,1)$ and $(3,1)$. Let $(u,v)$ be the pair in a B\'ezout tree corresponding to the relatively prime pair $(m,n)$ and $(U,V)$ be the pair given by the $\texttt{gcd}$ function for the same pair $(m,n)$. Then the following hold:
\begin{enumerate}
\item For all $(u,v)$ in the B\'ezout tree of $(3,1)$ generated by $(0,1)$, $(u,v) = (U,V)$. 
\item One third of the $(u,v)$ in the B\'ezout tree of $(2,1)$ generated by $(0,1)$ are not equal to $(U,V)$. Changing the value of $g(0,-1)$ in the second level of the B\'ezout tree from $(-1,2)$ to $(1,-1)$ results in a tree in which $(u,v) = (U,V)$ for all $(u,v)$.
\end{enumerate}
\end{conjecture}

\section{Application to mathematical art}

Our purpose for exploring trinary trees of relatively prime pairs was to use them to create hyperbolic wallpaper in a computationally inexpensive manner. These hyperbolic wallpapers are plotted in the upper-half plane and are created using a domain coloring algorithm. We start with a color map and a blank image. Then we use a mapping function $f(z)$ to color the pixels in the blank image, where $f(z)$ maps pixel locations from the blank image to pixel locations in the color map. This mapping function $f(z)$ is given by $$f(z) = \sum_{\gcd(c,d)=1}f(\gamma_{c,d}(z))$$ where $\gamma_{c,d}(z) = \frac{az+b}{cz+d}$ and $ad-bc = 1$. For our purposes, this sum will be finite yet have hundreds of terms. Notice that we need relatively prime pairs $(c,d)$ in our computation. The restriction $ad-bc = 1$ requires us to have the B\'ezout coefficients $(a,b)$ as well. Since the \texttt{gcd} function in \textsc{Matlab} is computationally expensive, we instead generate the trinary trees and B\'zout trees in a matter of seconds and pull values from these. More information on these wallpapers will be given in \cite{Handbook} upon publication.


\bibliographystyle{amsplain}

\end{document}